\documentclass{article}
\usepackage{amsmath}
\usepackage{amsthm}
\usepackage{amssymb}

\newcommand{\ds}{\displaystyle}
\newcommand{\vip}{\vskip0.15cm}
\newcommand{\ala}{\nonumber \\}
\newcommand{\indiq}{{\bf 1}}
\newcommand{\e}{{\varepsilon}}
\newcommand{\be}{{\bf e}}
\newcommand{\dtz}{{D_{\theta_0}}}

\newcommand{\cF}{{{\mathcal F}}}

\newcommand{\cA}{{{\mathcal A}}}

\newcommand{\rr}{{\mathbb{R}}}
\newcommand{\rd}{{\mathbb{R}^2}}
\newcommand{\rp}{{\mathbb{R}_+}}

\newcommand{\intot}{\ds\int_0^t}
\newcommand{\intrd}{\ds\int_{\rd} \!\!\!}
\newcommand{\intrp}{\ds\int_0^\infty \!\!\!}
\newcommand{\intpp}{\ds\int_{-\pi}^\pi \!}
\newcommand{\intpdp}{\ds\int_{-\pi/2}^{\pi/2} \!}

\newcommand{\cosa}{{\cos\alpha}}
\newcommand{\sina}{{\sin\alpha}}
\newcommand{\cost}{{\cos\theta}}
\newcommand{\sint}{{\sin\theta}}

\newtheorem{theo}{\indent Theorem}%[section]
\newtheorem{prop}[theo]{\indent Proposition}

\newtheorem{lem}[theo]{\indent Lemma}
\newtheorem{defin}[theo]{\indent Definition}

\begin{document}

\title{A new regularization possibility for the 
Boltzmann equation with soft potentials}

\author{Nicolas Fournier$^1$}

\maketitle

\footnotetext[1]{Centre de Math\'ematiques,
Facult\'e de Sciences et Technologies,
Universit\'e Paris~12, 61, avenue du G\'en\'eral de Gaulle, 94010 Cr\'eteil 
Cedex, France, {\tt nicolas.fournier@univ-paris12.fr}}

\begin{abstract}
\noindent We consider a simplified Boltzmann equation: spatially homogeneous, 
two-dimensional, radially symmetric, with Grad's angular cutoff, and
linearized around its initial condition.
We prove that for a sufficiently singular velocity cross section,
the solution may become instantaneously a function,
even if the initial condition is a singular measure.
To our knowledge, this is the first regularization result 
in the case with 
cutoff: all the previous results were relying on the non-integrability 
of the angular cross section.
Furthermore, our result is quite surprising: the regularization 
occurs for initial conditions that are not too singular, 
but also not too regular.
The objective of the present work is to explain that
the singularity of the velocity cross section, which is often considered
as a (technical) obstacle to regularization, seems on the contrary to 
help the regularization.
\end{abstract}

\maketitle



%%%%%%%%%%%%%%%%%%%%%%%%%%%%%%%%%%%%%%%%%%%%%%%%%%%%%%%%%%%%%%%%%
%%%%
%%%%            INTRODUCTION
%%%%
%%%%%%%%%%%%%%%%%%%%%%%%%%%%%%%%%%%%%%%%%%%%%%%%%%%%%%%%%%%%%%%%%

\section{Introduction} 

Let $f_t(dv)$ be the velocity distribution in a 2d spatially homogeneous 
dilute gas at time $t\geq 0$. Then under some physical assumptions,
$f$ solves the Boltzmann equation: for some $\gamma \in (-2,1]$, some
angular cross section $\beta$ (a nonnegative symmetric measure on 
$(-\pi,\pi)\backslash\{0\}$),
for all sufficiently regular functions $\varphi:\rd \mapsto\rr$,
\begin{equation}\label{wbe}
\frac{d}{dt} \intrd f_t(dv) \varphi(v) = \intrd f_t(dv)\intrd f_t(dv_*)
|v-v_*|^\gamma \intpp \beta(d\theta) [\varphi(v')-\varphi(v)],
\end{equation}
where, for $R_\theta$ the rotation of angle $\theta$,
\begin{equation}\label{dfvprime}
v'=v'(v,v_*,\theta)=\frac{v+v_*}{2} + R_\theta \frac{v-v_*}{2}.
\end{equation}
We refer to Villani \cite{villani} and Desvillettes \cite{desvillettes} 
for reviews 
on this equation. When $\gamma<0$, we speak of soft potentials.
The suject of the present paper is regularization: can 
$f_t$ be more regular than $f_0$, as soon as $t>0$~?

There are many results on this topic: it has been proved by several authors
that such a phenomenon occurs in the case without cutoff, that is
when $\beta(d\theta)=\beta(\theta)d\theta$ 
is sufficiently non-integrable near $\theta\sim 0$,
say $\beta(\theta) \sim \theta^{-1-\nu}$, with $\nu\in (0,2)$.
There has been essentially three types of results:

$\bullet$ it is shown in Desvillettes-Wennberg \cite{dw}
that in any dimension, when $\gamma\in [0,1)$, $\nu\in(0,2)$, 
and for an initial condition with finite mass, energy, entropy,
for a smoothed interaction kernel (i.e. when $|v-v_*|^\gamma$
is replaced by something like $(1+|v-v_*|^2)^{\gamma/2}$), 
then $f_t\in C^\infty$ for all $t>0$.
See also the (older) papers by Desvillettes \cite{d95}
and Alexandre-El Safadi \cite{al};

$\bullet$ in the case of Maxwellian molecules, i.e.
when $\gamma=0$, $\nu\in(0,2)$, 
Graham-M\'el\'eard \cite{gm} (see also \cite{f}) 
have proved that $f_t\in C^\infty$ for all $t>0$, 
even when the initial condition is a singular measure (but with many moments).
This works only in dimension $1$ or $2$.

$\bullet$ finally, the most general result (but also weaker) is the one of
Alexandre-Desvillettes-Villani-Wennberg \cite{advw}: in any dimension,
for any $\gamma$, any $\nu\in(0,2)$, any initial condition with finite
mass, energy, entropy, $\sqrt{f_t}$ instantaneously belongs to 
$H^{\nu/2}_{loc}$.
Let us observe that when tracking the constants in \cite{advw},
we realize that 
the result is weaker and weaker when $\gamma$ becomes more and more negative.
These results are based on lowerbounds of the entropy dissipation functionnal.
Such an idea was initiated by Lions \cite{l}, see also Villani \cite{vilreg}.

\vip

Our goal in the present paper is to show that  the explosion
of $|v-v_*|^\gamma$ near $v=v_*$, when $\gamma<0$, is not
an obstacle to regularization: ont the contrary, even in the case with cutoff,
such a singular interaction kernel may provide some regularization.

\vip

We start with a simplified equation, namely we linearize the Boltzmann equation
around its initial condition. We show that for a specific singular 
initial condition
(an uniform distribution on a circle), the solution instantaneously
becomes a function, in the case with cutoff, even if the angular cross
section is not a function.

\vip
The present result is
completely new, since all the previous results were
relying on the explosion near $0$ of the 
angular cross section $\beta$. 
It is furthermore 
very surprising, since, as we will show, no regularization may
hold if $f_0$ is too singular {\bf nor too regular}: if $f_0$
is a measure that is not a function but is too smooth in some sense,
then it will never become a function. This kind of phenomenon is fully
nonlinear. 

However, our result is very weak,
since it is qualitative (we
only prove that the solution immediately becomes a function), and
since we consider a simplified equation.

\vip 

Regularization is motivated by many other subjects
for which regularity estimates are needed: convergence
to equilibrium, uniqueness,... For example,
it is shown in \cite{fg} that uniqueness holds (with $\gamma<0$ and
a possibly non-cutoff angular cross-section)
for sufficiently smooth solutions (in some 
$L^p$, with $p$ large enough). We are far from such a 
quantitative regularization.

\vip

Let us finally mention that a similar result
should hold for the (nonlinear) Boltzmann equation.

\vip

{\bf Formal result.}
Assume that $\Lambda:=\beta((-\pi,\pi)\backslash\{0\})\in(0,\infty)$, and that
$\gamma\in (-2,-1)$. Consider $f_0$ the uniform law on the circle $\{|v|=1\}$.
Then for
any solution $(f_t)_{t\geq 0}$ to (\ref{wbe}), $f_t$ is a function
for a.e. $t>0$. 
(In dimension 3, take also $f_0$ uniform on $\{|v|=1\}$ 
but $\gamma\in(-3,-2)$).

\vip

{\it Formal proof.} We assume here that one may apply 
(\ref{wbe}) with any bounded measurable $\varphi$. Since
$f_0$ is a radially symmetric probability measure, 
so is $f_t$ for all $t\geq 0$.

\vip

{\it Step 1.} First, applying (\ref{wbe}) with $\varphi=\indiq_{\{0\}}$,
one easily deduces that $f_t(\{0\})=0$ for all $t\geq 0$. Indeed,
$f_0(\{0\})=0$, and simple considerations using that $f_t$ is radially symetric
show that for all $v$, all $\theta\ne 0$, 
$\indiq_{\{0\}}(v')-\indiq_{\{0\}}(v) \leq 0$ for $f_t$-a.e. $v_*$.

\vip

{\it Step 2.} Next, for any Lebesgue-null $A\subset \rd$, one gets
convinced that for all $\theta\ne 0$, $\indiq_A(v')=0$ for 
$f_t(dv)f_t(dv_*)$-a.e. $v,v_*$. Here, one has to use that $f_t$ is radially
symmetric and does not give weight to $0$.
As a consequence, 
\begin{equation}\label{forfor}
f_t(A) + \Lambda 
\intot ds \intrd f_s(dv)\intrd f_s(dv_*) |v-v_*|^\gamma\indiq_A(v)
\leq f_0(A).
\end{equation}
Thus if $f_0(A)=0$, we deduce that $f_t(A)=0$ for all $t\geq 0$.
This implies that $f_t(dv)$ has a density, except maybe on $C=\{|v|=1\}$. 

\vip

{\it Step 3.} It remains to check that $f_t(C)=0$ for a.e. $t>0$. But
(\ref{forfor}) applied with $C$ implies that 
\begin{equation*}
\int_0^\infty ds \intrd f_s(dv)\intrd f_s(dv_*) |v-v_*|^\gamma\indiq_C(v)
\leq \frac{f_0(C)}{\Lambda}=\frac{1}{\Lambda}.
\end{equation*}
As a consequence, for a.e. $t>0$, 
$\int\!\!\int_{C\times C} f_t(dv)f_t(dv_*) |v-v_*|^{\gamma}<\infty$.
This implies (see Falconer \cite[Theorem 4.13 p 64]{falconer}) that
either $f_t(C)=0$ or that the Haussdorff dimension of $C$ is greater than 
$|\gamma|$, the latter being excluded since $\dim_H(C)=1<|\gamma|$.
\qed

\vip

Unfortunately, we are not able to justify the use of such test functions
in the nonlinear case.

\vip

We state our result in Section \ref{res}, we prove it in Section \ref{reg}.
An appendix lies at the end of the paper.

\section{Main result}\label{res}

In the whole paper, 
the angular cross section is supposed to be finite, and to vanish on
$\{|\theta|\geq \pi/2\}$. For the nonlinear Boltzmann equation,
this last condition is not restrictive, for symmetrical reasons,
see the introduction of \cite{advw}. We also impose that $\beta$
vanishes near $0$ for simplicity.

\vip

{\bf Assumption $(A1)$:} $\beta$ is a nonnegative symmetric (even)
measure on $[-\pi/2,\pi/2]\backslash \{0\}$, with total mass 
$\Lambda=\beta([-\pi/2,\pi/2])\in (0,\infty)$.

{\bf Assumption $(A2)$:} there is $\theta_0 \in (0,\pi/2)$ 
such that $\beta((-\theta_0,\theta_0))=0$.

\vip

We now define the notion weak solutions
we will use. We denote by $Lip(\rd)$ the set of globally Lipschitz
functions from $\rd$ to $\rr$.

\begin{defin}
Let $\gamma \in (-2,0)$ be fixed, consider $\beta$ satisfying $(A1)$.
A family $(f_t)_{t\geq 0}$ of probability measures on $\rd$
is said to solve $LB(f_0,\gamma,\beta)$ if for all $\varphi \in Lip(\rd)$, 
all $t\geq 0$,
\begin{eqnarray}
&\intrd f_t(dv) \varphi(v) = \intrd f_0(dv) \varphi(v) + 
\intot ds \intrd f_s(dv) \intrd f_0(dv_*)
\cA \varphi (v,v_*),\label{lbe} \\
&\hbox{where} \quad
\cA \varphi(v,v_*) = \indiq_{\{v \ne v_*\}}
|v-v_*|^\gamma \intpdp \beta(d\theta)
\left[\varphi(v')-\varphi(v)\right] \label{dfA}
\end{eqnarray}
with $v'=v'(v,v_*,\theta)$ defined in (\ref{dfvprime}).
\end{defin}

We will check later that in our situation, all the terms make sense
in (\ref{lbe}). The indicator $v\ne v_*$ is written for convenience, since
$v=v_*$ implies $v'=v$. Our main result writes as follows.

\begin{theo}\label{main}
Let $\gamma \in (-2,-1)$ be fixed, consider $\beta$ satisfying $(A1-A2)$.
Assume that for some $r_0>0$, $f_0$ is a uniform distribution on the circle
$\{|v|=r_0\}$. Then the exists a solution $(f_t)_{t\geq 0}$ to
$LB(f_0,\gamma,\beta)$ 
such that $f_t$ has a density w.r.t. the Lebesgue measure on
$\rd$ for almost every $t>0$.
\end{theo}

Let us comment on this result.
Consider $\gamma\in(-2,0)$, and an initial condition $f_0$ satisfying
\begin{equation}\label{ffini}
\hbox{for } f_0 \hbox{-a.e. } v\in\rd, \quad 
\lambda_0(v):=\intrd f_0(dv_*) |v-v_*|^\gamma \indiq_{\{v_* \ne v\}} <\infty.
\end{equation}
Then we consider a solution $(f_t)_{t\geq 0}$ to $LB(f_0,\gamma,\beta)$. 
Applying (\ref{lbe}) with some nonnegative Lipschitz 
function $\varphi$ and using $(A1)$, we immediately get
\begin{equation*}
\frac{d}{dt}\intrd f_t(dv) \varphi(v) \geq
- \Lambda \intrd f_t(dv) \lambda_0(v) \varphi(v),
\end{equation*}
whence, at least formally,
$f_t(dv) \geq e^{-\Lambda t \lambda_0(v)} f_0(dv)$ for all $t\geq 0$:
no regularization may occur.

\vip

This result is fully nonlinear and quite surprising: if
$f_0$ is regular enough to satisfy (\ref{ffini}) but is not a function,
then it does never become a function. Such examples can easily be
built: as shown in Falconer \cite[Theorem 4.13 p 64]{falconer},
for any Borel subset $A\subset\rd$ with Haussdorff dimension strictly greater
than $|\gamma|$, we may find a probability measure $f_0$ on $\rd$ with 
$f_0(A)=1$ and such that
such that $\int_\rd f_0(dv)\int_\rd f_0(dv_*) |v-v_*|^\gamma<\infty$,
which of course imply (\ref{ffini}).

Let us insist on the fact that initial conditions satisfying (\ref{ffini})
are more regular than the uniform distribution on $\{|v|=1\}$:
the latter gives positive weight to some sets with lower dimension.

Note also that on the contrary, $f_0$ has to be sufficiently regular.
If for example we assume that $f_0=\frac{1}{2} (\delta_{v_0}+\delta_{v_1})$,
no regularization may hold (due to the indicator function in (\ref{ffini})).
The same argument applies to $f_0=\frac{1}{2}(\delta_{v_0}+g_0)$, 
for some bounded probability density $g_0$.

\vip

Finally, let us mention that our result holds even if
$\beta= \delta_{\theta_0}+\delta_{-\theta_0}$, for some
fixed $\theta_0\in (0,\pi/2)$: the regularization does really
not follow from the regularity of the angular cross section.

\section{Proof}\label{reg}

The aim of this section is to prove Theorem \ref{main}.
We assume in the whole section that
$(A1-A2)$ hold, 
that $r_0=1$ (for simplicity), and that $\gamma\in (-2,-1)$ is fixed.
Thus our initial condition $f_0$ is defined by 
\begin{equation}\label{cini}
\int_\rd \varphi(v)f_0(dv)= \frac{1}{2\pi} \int_{-\pi}^\pi 
d\alpha \varphi(\be_\alpha),
\end{equation}
for all measurable $\varphi:\rd\mapsto \rr_+$, 
where $\be_\alpha:=(\cos \alpha, \sin \alpha)$.

\begin{prop}\label{exist}
$(i)$ There exists a radially symmetric solution $(f_t)_{t\geq 0}$
to $LB(f_0,\gamma,\beta)$. 

(ii) For $t\geq 0$, define the probability measure
$\lambda_t$ on $\rr_+$ by $\lambda_t(A)=f_t(\{|v|\in A\})$. 
Then we have, for all $\varphi \in L^\infty(\rd)$, all $t\geq 0$, 
\begin{equation}\label{rdrdtheta}
\int_{\rd} \varphi(v) f_t(dv) = \frac{1}{2\pi} \int_{-\pi}^\pi d\alpha
\int_0^\infty \lambda_t(dr) \varphi(r\be_\alpha).
\end{equation}
(iii) Consider the class $\cF$ of functions of the form $\psi=\varphi+ \delta$,
where $\varphi\in Lip(\rr_+)$, 
and with $\delta
\in L^\infty(\rr)$, with $1\notin$ {\rm supp} $\delta$. Then
\begin{equation}\label{eqr}
\intrp \lambda_t(dr) \psi(r) = \psi(1) + \intot ds \intrp \lambda_s(dr) 
\intpdp \beta(d\theta)\intpp \frac{d\alpha}{2\pi} 
(r^2+1-2r\cosa)^{\gamma/2} [\psi(r')-\psi(r) ]
\end{equation}
where $r'=r'(r,\theta,\alpha)=
\left(\frac{1+\cost}{2}r^2 + \frac{1-\cost}{2} 
-r \sint \sina \right)^{1/2}$.
\end{prop}

To understand (\ref{eqr}), observe that 
$r'(r,\theta,\alpha)=|v'(r\be_{\alpha_0+\alpha},
\be_{\alpha_0},\theta)|$ and that
$(r^2+1-2r\cosa)^{\gamma/2}=|r\be_{\alpha_0+\alpha}-\be_{\alpha_0} |^\gamma$
for any $\alpha_0$.
This result is routine and will be checked at the end of the section. 
We now start the proof of Theorem \ref{main}. First, 
$0$ cannot be reached by the radius distribution 
$\lambda_t$.

\begin{lem}\label{p0}
Consider the family $(\lambda_t)_{t\geq 0}$ introduced in Proposition 
\ref{exist}. For all $t\geq 0$, $\lambda_t(\{0\})=0$.
\end{lem}

\begin{proof} Since $\indiq_{\{0\}}$ belongs to $\cF$, we may apply
(\ref{eqr}). We realize that 

(a) initially, $\indiq_{\{0\}}(1)=0$; 

(b) (when $r=0$) $\indiq_{\{0\}}(r'(0,\theta,\alpha))-\indiq_{\{0\}}(0)\leq 0$ 
for all $\alpha,\theta$;

(c) when $r>0$, for all $\theta\in [-\pi/2,\pi/2]\backslash
\{0\}$, $\indiq_{\{0\}}(r'(r,\theta,\alpha))=0$ 
for $d\alpha$-a.e. $\alpha\in [-\pi,\pi]$ (use here that $r\sint \ne 0$).

Hence (\ref{eqr}) yields that for all $t\geq 0$,
$\int_0^\infty \lambda_t(dr) \indiq_{\{0\}}(r) \leq 0$, 
and the result follows.
\end{proof}

Then we may prove that the radius distribution $\lambda_t$ has a density,
except maybe at $1$.

\begin{lem}\label{densp1}
Consider the family $(\lambda_t)_{t\geq 0}$ introduced in Proposition 
\ref{exist}. Consider a Lebesgue-null subset $A\subset \rp$ 
with $1\notin A$. For all $t\geq 0$, $\lambda_t(A)=0$.
\end{lem}

\begin{proof}
We first assume that $1\notin \bar A$.
Then $\indiq_A$ belongs to $\cF$, so that we may use (\ref{eqr}).
Since initially 
$\indiq_A(1)=0$ and due to Lemma \ref{p0}, it suffices to prove that
for all $r>0$, for all $\theta \in [-\pi/2,\pi/2]\backslash
\{0\}$, $\indiq_A(r'(r,\theta,\alpha))=0$ for $d\alpha$-a.e. 
$\alpha\in [-\pi,\pi]$. 
But this is immediate, using that $r\sint \ne 0$, that $A$ is Lebesgue-null,
and the substitution $\alpha \mapsto r'(r,\theta,\alpha)$.

As previously, (\ref{eqr}) yields that for all $t\geq 0$,
$\int_0^\infty \lambda_t(dr) \indiq_A(r) \leq 0$, 
and the result follows.
\vip

Now if $1\in \bar A$, we consider $A_n=A\cap \{|r-1|\geq 1/n\}$,
which increases to $A$ (because $1\notin A$ by assumption). 
Since $1\notin \bar A_n$,
we know that
for all $t\geq 0$, all $n\geq 1$, $\lambda_t(A_n)=0$. Making $n$ tend to 
infinity, we get $\lambda_t(A)=0$
for all $t\geq 0$ by the Beppo-Levi Theorem.
\end{proof}

Finally, we prove that our solution leaves immediately the unit circle.

\begin{lem}\label{p1}
Consider the family $(\lambda_t)_{t\geq 0}$ introduced in Proposition 
\ref{exist}. 
For a.e. $t> 0$, $\lambda_t(\{1\})=0$.
\end{lem}

\begin{proof} We divide the proof into two steps.
\vip
{\it Step 1.} We first show that for all $T>0$, there is
$\kappa_T>0$ such that for all $\e>0$ small enough,
\begin{equation}\label{ob1}
\int_0^T ds \lambda_s(\{1\}) \leq \kappa_T \e^{|\gamma|-1} + \kappa_T
\e^{|\gamma|}
\int_0^T ds \intrp \lambda_s(dr) |r^2-1|^\gamma \indiq_{\{|r^2-1| > \e\}}.
\end{equation}
We consider $\e \in (0,1)$, and we apply
(\ref{eqr}) with $\psi(r)=\indiq_{\{|r^2-1|\leq \e\}}$, which
belongs to $\cF$. We get
\begin{equation}\label{jab1}
\lambda_0(\{|r^2-1|\leq \e\}) - \lambda_T(\{|r^2-1|\leq \e\})
= \int_0^T ds \intrp \lambda_s(dr) [A_\e(r) - B_\e(r)],
\end{equation}
where, using $(A1-A2)$ and setting $\dtz=[-\pi/2,\pi/2] \backslash 
(-\theta_0,\theta_0)$,
\begin{eqnarray*}
A_\e(r)&=& \indiq_{\{|r^2-1|\leq \e\}} \int_\dtz \beta(d\theta) \intpp 
\frac{d\alpha}{2\pi} \indiq_{\{|r'^2-1|> \e\}} (r^2+1-2r\cosa)^{\gamma/2},
\ala
B_\e(r)&=& \indiq_{\{|r^2-1| > \e\}} \int_\dtz \beta(d\theta) \intpp 
\frac{d\alpha}{2\pi} \indiq_{\{|r'^2-1|\leq \e\}} (r^2+1-2r\cosa)^{\gamma/2}.
\end{eqnarray*}
We first give a lowerbound of $A_\e$. We only consider the case
where $r=1$.
Since $|(r'(1,\theta,\alpha))^2-1|=|\sint\sina|\geq (\sin{\theta_0})
\alpha/2$, for $\alpha \in (0,\pi/2)$, and since
$(2-2\cosa) \leq \alpha^2$,
we obtain
\begin{eqnarray}\label{jab2}
A_\e(r) \geq \indiq_{\{r=1\}} \frac{\beta(\dtz)}{2\pi}
\int_0^{\pi/2} d\alpha \indiq_{\{\alpha>\frac{2\e}{\sin\theta_0}\}} 
\alpha^{\gamma} \geq c_0  \indiq_{\{r=1\}} \e^{\gamma+1}
\end{eqnarray}
for some constant $c_0>0$, at least for $\e \in (0,\e_0)$
with $\e_0:=(\sin\theta_0)\pi/8$.

We now upperbound $B_\e$. Some easy considerations allow us to get
\begin{eqnarray*}
B_\e(r) \leq \frac{2}{\pi}\indiq_{\{|r^2-1| > \e\}} 
\int_{\theta_0}^{\pi/2} \!\! \beta(d\theta)
\int_{-\pi/2}^{\pi/2} \! \!d\alpha 
\indiq_{\{|r'^2-1|\leq \e\}} (r^2+1-2r\cosa)^{\gamma/2}.
\end{eqnarray*}
We of course have $(r^2+1-2r\cosa)^{\gamma/2} \leq |r-1|^\gamma$, and
a computation shows that $|r'^2-1|\leq \e$ implies that
$\sina \in [\frac{(1+\cost)(r^2-1) - 2\e}{2r\sint},\frac{(1+\cost)(r^2-1) 
+ 2\e}{2r\sint} ]$. This yields
\begin{eqnarray*}
B_\e(r) \leq \frac{2}{\pi}|r-1|^\gamma \indiq_{\{|r^2-1| > \e\}} 
\int_{\theta_0}^{\pi/2} \!\! \beta(d\theta)
\int_{-\pi/2}^{\pi/2} \! \!d\alpha 
\indiq_{\{\sina \in [\frac{(1+\cost)(r^2-1) \pm 2\e}{2r\sint}]\}}.
\end{eqnarray*}
Consider first $a_1>0$ such that $|r^2-1|\leq a_1$ implies $r\in[1/2,2]$
and $\frac{(1+\cost)(r^2-1)}{2r\sint} \in [-\pi/8,\pi/8]$ 
for all $\theta\in [\theta_0,\pi/2]$.

Consider $\e_1>0$ such $|r^2-1|\leq a_1$ implies
$[\frac{(1+\cost)(r^2-1) \pm 2\e_1}{2r\sint}] \subset [-\pi/4,\pi/4]$
for all $\theta\in [\theta_0,\pi/2]$.

Then for $|r^2-1|\leq a_1$, we get, for all $\e\in(0,\e_1)$, 
for some constants $\kappa_1,c_1>0$, 
(since then $2r\sin\theta\geq
\sint_0>0$ and $|r^2-1|=|r-1|(r+1)\leq 3 |r-1|$),
\begin{equation}\label{jab3}
B_\e(r) \leq \kappa_1 \e |r-1|^\gamma \indiq_{\{|r^2-1| > \e\}}
\leq c_1 \e |r^2-1|^\gamma \indiq_{\{|r^2-1| > \e\}}.
\end{equation}
On the other hand, it is immediate that for some
$c_2>0$, for $|r^2-1|\geq a_1$ (so that $|r-1|\geq a_2>0$),
\begin{equation}\label{jab4}
B_\e(r) \leq \frac{2}{\pi} \beta([\theta_0,\pi/2]) \pi |r-1|^{\gamma} 
\leq c_2.
\end{equation}
Using that $\lambda_0(\{|r^2-1|\leq \e\})-\lambda_T(\{|r^2-1|\leq \e\})\leq1$
and gathering 
(\ref{jab1}-\ref{jab2}-\ref{jab3}-\ref{jab4}), we obtain for 
all $\e \in (0,\e_2)$, with $\e_2 = \min(\e_0,\e_1)$,
\begin{eqnarray*}
1\geq \int_0^T ds \intrp\lambda_s(dr) \left[
c_0\indiq_{\{r=1\}} \e^{\gamma+1} 
-c_1 \e |r^2-1|^\gamma \indiq_{\{|r^2-1| > \e\}} - c_2
\right],
\end{eqnarray*}
whence (\ref{ob1}).
\vip
{\it Step 2.} We now conclude.
Consider the measure $\mu_T(dr)=\int_0^T ds \lambda_s(dr)\indiq_{\{r\ne 1\}}$.
Since $\mu_T$ is finite and $\mu_T(\{1\})=0$, 
the de la Vall\'ee Poussin Lemma 
\ref{lvp} ensures us that 
there exists a  
function $g: \rr_+ \mapsto \rr_+$, with $g(\infty)=\infty$,
such that $x \mapsto x g(1/x)$ is nondecreasing on $\rr_+$, and $\int_0^\infty
\mu_T(dr) g(1/|r^2-1|) < \infty$. Since $|\gamma| \geq 1$ by assumption,
we deduce that 
\begin{equation*}
\indiq_{\{|r^2-1|>\e\}}
\leq \frac{|r^2-1|^{|\gamma|} g(1/|r^2-1|)}{\e^{|\gamma|} g(1/\e)}, 
\end{equation*}
so that
(\ref{ob1}) becomes
\begin{eqnarray*}
\int_0^T ds \lambda_s(\{1\}) \leq \kappa_T \e^{|\gamma|-1} +
\frac{\kappa_T}{g(1/\e)} \int_0^\infty \mu_T(dr) g(1/|r^2-1|).
\end{eqnarray*}
Letting $\e$ tend to $0$, we get $\int_0^T ds \lambda_s( \{1\})=0$.
Since $T$ is arbitrarily large, this ends the proof.
\end{proof}

We may now conclude the

\vip

{\it Proof of Theorem \ref{main}.}
We consider the solution $(f_t)_{t\geq 0}$ built in Proposition \ref{exist},
and the associated radius distribution $(\lambda_t)_{t\geq 0}$.
Owing to the Radon-Nikodym Theorem, to Lemmas \ref{densp1} and \ref{p1}, 
we deduce that for a.e. $t\geq 0$,
$\lambda_t(dr)$ has a density $\lambda_t(r)$ with respect to the 
Lebesgue measure on $\rr_+$.
Then we deduce from (\ref{rdrdtheta}) 
that $f_t(dv)$ has the density 
$f_t(v)=\lambda_t(|v|)/(2\pi|v|)\indiq_{\{|v|\ne 0\}}$, 
and thus is indeed a function.
(The case $v=0$ is not a problem, since $f_t(\{0\})=\lambda_t(\{0\})=0$).
\qed

\vip

We conclude the section with the

\vip

{\it Proof of Proposition \ref{exist}.} We split the proof into 4 steps.
\vip
{\it Step 1.} We first check the existence of a solution.
We introduce, for $n\geq 1$, the operator $\cA_n$, of which the expression
is the same as (\ref{dfA}) with $\min(|v-v_*|^\gamma,n)$
instead of $|v-v_*|^\gamma$. Then we observe that with
our choice for $f_0$, as shown in the appendix, we have
for all $\varphi \in Lip(\rd)$,
\begin{eqnarray}
&\left|\int_\rd f_0(dv_*)\cA\varphi(v,v_*)\right| \leq  C(\Lambda,\gamma) 
||\varphi||_{lip}, \label{ms} \\
&v\mapsto \int_\rd f_0(dv_*)\cA\varphi(v,v_*)
\hbox{ is continuous on } \rd, \label{ms2} \\
&\sup_{n\geq 1} \left|\int_\rd f_0(dv_*)\cA_n\varphi(v,v_*)\right| 
\leq  C(\Lambda,\gamma) ||\varphi||_{lip}, \label{ms3}\\
& \lim_{n \to \infty} \sup_{v\in\rd}
\left| \int_\rd f_0(dv_*)(\cA-\cA_n)\varphi(v,v_*)\right| =0.\label{ms4}
\end{eqnarray}
In particular, (\ref{ms}) implies that all the terms make sense in (\ref{lbe}).

One easily checks, 
by classical methods (Gronwall Lemma and Picard iteration
using the total variation norm), that there exists a unique
solution $(f^n_t)_{t\geq 0}$ to $LB_n(f_0,\gamma,\beta)$,
where $\cA$ is replaced by $\cA_n$. The obtained solution $f^n$ 
is clearly radially symmetric.

Next, using (\ref{lbe}) and (\ref{ms3}), we deduce that 

(a) using $\varphi(v)=|v|$, 
$C_T:=\sup_n \sup_{[0,T]} \int_{\rd} f^n_t(dv)
|v| \leq 1+C(\Lambda,\gamma) T$ for all $T>0$,

(b) for any $\varphi\in Lip(\rd)$,
for all $0\leq s \leq t$, $\left|\int_{\rd} (f^n_t-f^n_s)(dv)\varphi(v) \right|
\leq C(\Lambda,\gamma) ||\varphi||_{lip} |t-s|$.

Point (a) ensures that for each $t\geq 0$, $(f^n_t)_{n\geq 1}$ is tight,
while (b) gives some equicontinuity estimates. It is then standard that
up to extraction of a 
(not relabelled) subsequence, $(f^n_t)_{t\geq 0}$ tends to some 
family of (radially symmetric) probability measures 
$(f_t)_{t\geq 0}$, 
in the sense that for all $\varphi \in Lip(\rd)$,
for all $T\geq 0$,  $\lim_n\sup_{[0,T]} | \int_{\rd} 
(f^n_t-f_t)(dv)\varphi(v)  | =0$. This also implies
that for all $t\geq 0$, all $\varphi:\rd\mapsto\rr$ 
continuous and bounded, $\lim_n \int_{\rd} 
f^n_t(dv)\varphi(v)= \int_{\rd} f_t(dv)\varphi(v)$.
We deduce that $(f_t)_{t\geq 0}$ solves $LB(f_0,\gamma,\beta)$, by 
passing to the limit in $LB_n(f_0,\gamma,\beta)$, using the 
convergence properties of $f^n$ to $f$, the Lebesgue dominated convergence
Theorem, as well as 
(\ref{ms}-\ref{ms2}-\ref{ms3}-\ref{ms4}).

\vip
{\it Step 2.} Next, point (ii) 
of the statement is a simple consequence of the
radial symmetry of $(f_t)_{t\geq 0}$.
\vip
{\it Step 3.} We now check point (iii) when 
$\psi \in Lip(\rr)$. Then $\varphi(v)=\psi(|v|)\in Lip(\rd)$, 
and we thus may apply (\ref{lbe}). Using
several times (\ref{rdrdtheta}) 
and the expression (\ref{cini}) of $f_0$, we get
\begin{eqnarray}\label{qq1}
\intrp \lambda_t(dr) \psi(r) = \intrd f_t(dv) \varphi(v) 
= \intrd f_0(dv) \varphi(v) + \intot ds \intrd f_s(dv) \intrd f_0(dv_*)
\cA\varphi(v,v_*)\ala
=  \psi(1) + \intot ds \intrp \lambda_s(dr) 
\intpp \frac{d\alpha_*}{2\pi}\intpp
\frac{d\alpha}{2\pi} \cA\varphi(r \be_\alpha,\be_{\alpha_*}).
\end{eqnarray}
Then a simple computation shows that 
\begin{eqnarray}\label{qq2}
\intpp \frac{d\alpha}{2\pi}
\cA\varphi(r \be_\alpha,\be_{\alpha_*})&=&\intpdp \beta(d\theta) 
\intpp \frac{d\alpha}{2\pi} 
|r\be_\alpha - \be_{\alpha_*}|^\gamma 
[\psi(|v'(r \be_\alpha,\be_{\alpha_*},\theta)|) - \psi(|r \be_\alpha|)]\ala
&\hskip-1.3cm =&\hskip-1cm \intpdp \beta(d\theta)  \intpp 
\frac{d\alpha}{2\pi}
(r^2+1-2r\cos(\alpha-\alpha_*))^{\gamma/2} 
[\psi(r'(r,\theta,\alpha-\alpha_*)
-\psi(r)]\ala
&\hskip-1.3cm =&\hskip-1cm\intpdp \beta(d\theta)  \intpp 
\frac{d\alpha}{2\pi} (r^2+1-2r\cosa)^{\gamma/2} 
[\psi(r'(r,\theta,\alpha)
-\psi(r)],
\end{eqnarray}
and in particular does not depend on $\alpha_*$. Gathering (\ref{qq1}) and 
(\ref{qq2}), we obtain (\ref{lbe}).
\vip
{\it Step 4.}
Finally, we have to prove that (\ref{lbe}) still holds when 
$\psi\in L^\infty([0,\infty))$, such that there exists $\e>0$ with 
$\psi=\psi\indiq_{\{|r^2-1|\geq 2\e\}}$.

To this end, we consider the finite Borel measure $\mu_t$ on $\rr_+$ defined by
$\mu_t(A):=\lambda_t(A) + \int_0^t ds  
[\lambda_s(A)+\int \beta(d\theta) \int_{-\pi}^\pi d\alpha 
\lambda_s(\{r'\in A\})] $.
We consider  $\psi_n \in Lip(\rr_+)$,
uniformly bounded by $2||\psi||_{\infty}$, 
satisfying $\psi_n=\psi_n\indiq_{\{|r^2-1|\geq \e\}}$
and such that
$\psi_n(r)$ tends to $\psi(r)$ for $\mu_t$-a.e. $r\in [0,\infty)$.
Such an approximating sequence can be found, due to the Lusin Theorem
(see e.g. Rudin \cite{rudin}).

Then we may apply (\ref{eqr}) for each $n\geq 1$, and get
\begin{eqnarray*}
\intrp \lambda_t(dr) \psi_n(r) = \intot ds \intrp \lambda_s(dr) 
\intpdp \beta(d\theta)
\intpp \frac{d\alpha}{2\pi} (r^2+1-2r\cosa)^{\gamma/2} 
[\psi_n(r')-\psi_n(r) ].
\end{eqnarray*}
To pass to the limit in this equation, we will use the Lebesgue
dominated
convergence Theorem. First, $\psi_n$ is uniformly bounded, so that the
left hand side is not a problem (recall that $\psi_n$ goes to $\psi$ 
$\mu_t$-a.e., and thus $\lambda_t$-a.e.). Next,
using the properties of $\psi_n$, we get
$|\psi_n(r')-\psi_n(r)| \leq  c
[\indiq_{\{|r'^2-1|\geq \e\}}+\indiq_{\{|r^2-1|\geq \e\}} ]$
and the proof will be finished 
if we show that
\begin{equation*}
O_\e(t)=\intot ds \intrp \lambda_s(dr) 
\intpdp \beta(d\theta)
\intpp \frac{d\alpha}{2\pi} (r^2+1-2r\cosa)^{\gamma/2}
(\indiq_{\{|r'^2-1|\geq \e\}}+\indiq_{\{|r^2-1|\geq \e\}})<\infty,
\end{equation*}
because since $\psi_n$ goes to $\psi$ $\mu_t$-a.e., 
$[\psi_n(r')-\psi_n(r)]$ goes to $[\psi(r')-\psi(r)]$, $ds\lambda_s(dr)
d\alpha\beta(d\theta)$-a.e.

To show that $O_\e(t)<\infty$,  we observe that $|r'^2-1|=
|(1+\cost)(r^2-1) - 2r\sint\sina|/2 \leq |r^2-1|+r|\sina|$.
Thus $|r'^2-1|\geq \e$ implies that 
either $|r^2-1|\geq \e/2$ or $r|\sina|\geq \e/2$.
Hence, recalling $(A1)$ and since $(r^2+1-2r\cosa)^{\gamma/2} \leq 
((r-1)^2+ 2r \alpha^2/5)^{\gamma/2}$, 
(because $1-\cosa \geq \alpha^2/5$ on $[-\pi,\pi]$),
\begin{equation*}
O_\e(t) \leq \Lambda t \sup_{r\geq 0, |\alpha|\leq \pi}
(2 \indiq_{\{|r^2-1|\geq \e/2\}}+ 
\indiq_{\{r|\sina|\geq \e/2\}})((r-1)^2+ 2r \alpha^2/5)^{\gamma/2} .
\end{equation*}
This last quantity is bounded for each $\e>0$ fixed 
(separate the cases $\{|r^2-1|\geq \e/2\}$, $\{r>2\}$, and 
$\{r\leq 2,
r|\sin\alpha|\geq \e/4\}\subset\{2r\alpha^2/5\geq \e^2/80\}$).
This concludes the proof.\qed

\section{Appendix}\label{ap}

We start with a result in  the spirit of de la Vall\'ee Poussin,  
adapted to our problem.

\begin{lem}\label{lvp}
Let $\mu$ be a nonegative finite measure on $\rr_+$
such that $\mu(\{1\})=0$. Then there exists a
function $g:\rr_+\mapsto \rr_+$
such that $\lim_{\infty} g=\infty$ and 
$\int_0^\infty \mu(dr)g(1/|r^2-1|)<\infty$. Furthermore, $g$ can be chosen 
in such a way that $x\mapsto xg(1/x)$ is non-decreasing on $\rr_+$.
\end{lem}

\begin{proof}
Since $\mu$ is finite and since $\mu(\{1\})=0$, we may find an 
increasing sequence $(a_k)_{k\geq 1}\subset (0,\infty)$, with 
$\mu(\{|r^2-1|\leq 1/a_k\}) \leq 2^{-k}$. We also set $a_0=0$, and
define the non-decreasing function $f:\rr_+\mapsto [1,\infty)$ by
$f(x)=k+1$ if $x\in [a_k,a_{k+1})$. Then
$\lim_{\infty} f = \infty$, and
\begin{equation*}
\int_0^\infty \mu(dr) f(1/|r^2-1|) = \sum_{k\geq 0}
(k+1) \mu(\{|r^2-1| \in (1/a_{k+1},1/a_k]\})
\leq \mu(\rr_+)+\sum_{k\geq 1} (k+1)2^{-k}<\infty.
\end{equation*}
We now set $g(x):=x \inf_{[0,x]} (f(y)/y)\leq f(x)$.
Hence $\int_0^\infty \mu(dr)g(1/|r^2-1|)<\infty$.
Moreover, $xg(1/x)=\inf_{[0,1/x]} \frac{f(y)}{y}$ is clearly 
non-decreasing. We finally have to check that $\lim_\infty g =\infty$.
But for each $\e>0$, for all $x\geq 0$, 
$$
g(x) \geq \min(x\inf_{[0,\e x]} f(y)/y,
x\inf_{[\e x,x]} f(y)/y) \geq \min(1/\e,\inf_{[\e x,x]}f(y))
$$
since $f\geq 1$. Using that $\lim_\infty f = \infty$,
we obtain $\liminf_\infty g \geq 1/\e$. This holds for all $\e>0$,
and thus allows us to conclude.
\end{proof}

Before proving (\ref{ms}-\ref{ms2}-\ref{ms3}-\ref{ms4}), we observe that
for $\be_\alpha=(\cos \alpha, \sin \alpha)$, the function
\begin{equation*}%
h_\delta(v):=\frac{1}{2\pi}\int_{-\pi}^\pi d\alpha |v-\be_\alpha|^\delta
\end{equation*}
is bounded on $\rd$ if $\delta\in (-1,0)$.

\vip

{\it Proof of (\ref{ms}).} We first observe, recalling 
(\ref{dfvprime}), that $|v'-v| \leq  |v-v_*|$.
Thus, due to $(A1)$, (\ref{dfA}) and (\ref{cini}),
we deduce that 
$|\int_\rd f_0(dv_*) \cA \varphi(v,v_*) |\leq \Lambda ||\varphi||_{lip}
\int_\rd f_0(dv_*)|v-v_*|^{\gamma+1}=h_{\gamma+1}(v)$, which is bounded
since $\gamma+1 \in (-1,0)$ by assumption.
\qed

\vip

{\it Proof of (\ref{ms3}).} It is the same as that of (\ref{ms}).
\qed

\vip

{\it Proof of (\ref{ms4}).} Using the same arguments as in the proof
of (\ref{ms}), 
we get 
\begin{eqnarray*}
|\intrd f_0(dv_*) (\cA-\cA_n) \varphi(v,v_*)| \leq
\Lambda ||\varphi||_{lip}
\intrd f_0(dv_*) |v-v_*|^{\gamma+1}\indiq_{\{|v-v_*|^\gamma \geq n\}}\ala
\leq \Lambda ||\varphi||_{lip}\intrd f_0(dv_*) 
|v-v_*|^{\gamma/2} |v-v_*|^{1+\gamma/2}\indiq_{\{|v-v_*| \leq n^{1/\gamma}\}}
\ala
\leq \Lambda ||\varphi||_{lip} n^{(2+\gamma)/2\gamma} \intrd f_0(dv_*)
|v-v_*|^{\gamma/2}=\Lambda ||\varphi||_{lip} n^{(2+\gamma)/2\gamma}
h_{\gamma/2}(v).
\end{eqnarray*}
We used here that $1+\gamma/2>0$. Since $h_{\gamma/2}$ is bounded
(because $\gamma/2\in(-1,0)$), and since 
$ (2+\gamma)/2\gamma <0$, the result follows.
\qed

\vip

{\it Proof of (\ref{ms2}).} For $\varphi\in Lip(\rd)$, 
we set $h_\varphi(v)=\int_\rd f_0(dv_*) \cA\varphi(v,v_*)$.
We wish to show that $h_\varphi$ is continuous on $\rd$. This
follows from (\ref{ms4}). Indeed, consider $h_\varphi^n$, where
$\cA$ is replaced by $\cA_n$. Then for each $n\geq 1$, 
$h_\varphi^n$ is obviously
continuous on $\rd$, by the Lebesgue Theorem
(because for all $v_*,\theta$, $v\mapsto \min(|v-v_*|^\gamma,n)$ and
$v\mapsto v'(v,v_*,\theta)$ are continuous). But (\ref{ms4}) implies
that $h_\varphi^n$ goes uniformly to $h_\varphi$ on $\rd$.
\qed

\vip

{\bf Acknowledgements} I wish to thank Jacques Printems for
stimulating discussions.

\def\refname{References}

\end{document}